\documentclass[reqno, 11pt]{amsart}
\usepackage[utf8]{inputenc}
\usepackage{amsmath}
\usepackage{amsthm}
\usepackage{amssymb}
\usepackage{mathtools}
\usepackage{booktabs}
\usepackage{amsxtra}
\usepackage{amsbsy}
\usepackage{comment}
\usepackage{mathrsfs}
\usepackage{tabularx,ragged2e,booktabs,caption}
\usepackage{longtable}
\usepackage{lipsum}

\usepackage[english]{babel}
\usepackage{blindtext}
\usepackage{url}
\usepackage{bbm}
\usepackage{euscript}
\usepackage{pifont}
\usepackage{hyperref}
\usepackage{wasysym}
\hypersetup{
    colorlinks=true,
    linkcolor=blue,
    filecolor=magenta,
    urlcolor=cyan,
}
\usepackage{graphicx}
\usepackage{epstopdf}

    \usepackage[
    top    = 2.60cm,
    bottom = 2.60cm,
    left   = 2.9 cm,
    right  = 2.9 cm]{geometry}

\newtheorem{thm}{Theorem}[section]
\newtheorem{lem}[thm]{Lemma}
\newtheorem{cor}[thm]{Corollary}
\newtheorem{prop}[thm]{Proposition}

\newtheorem*{conjecture*}{Conjecture}
\newtheorem*{thm*}{Theorem}

\theoremstyle{remark}

\theoremstyle{definition}

\newcommand{\Out}{\operatorname{Out}}

\newcommand{\Syl}{\operatorname{Syl}}
\newcommand{\cha}{\operatorname{char}}
\newcommand{\GL}{\operatorname{GL}}
\newcommand{\Al}{\textup{\textsf{A}}}
\newcommand{\Sy}{\textup{\textsf{S}}}

\newskip\aline \newskip\halfaline
\def\skipaline{\vskip\aline}

\def\qedbox{$\rlap{$\sqcap$}\sqcup$}
\def\qed{\nobreak\hfill\penalty250 \hbox{}\nobreak\hfill\qedbox\skipaline}

\title[On the odd order composition factors]{On the odd order composition factors of finite linear groups}

\author[A. Betz]{Alexander Betz}
\author[M. Chao-Haft]{Max Chao-Haft}
\author[T. Gong]{Ting Gong}
\author[A. Ter-Saakov]{Anthony Ter-Saakov}
\author[Y. Yang]{Yong Yang}
\address{Department of Mathematics, Le Moyne College, 1419 Salt Springs Road, Syracuse, NY 13214}
\email{betzas@lemoyne.edu}
\address{Department of Mathematics, Harvey Mudd College, 340 East Foothill Boulevard, CA 91711}
\email{mchaohaft@g.hmc.edu}
\address{Department of Mathematics, University of Notre Dame, 255 Hurley, Notre Dame, IN 46556}
\email{tgong@nd.edu}
\address{Department of Mathematics $\&$ Statistics, Boston University, 111 Cummington Mall, Boston, MA 02215}
\email{antter@bu.edu}
\address{Department of Mathematics, Texas State University, 601 University Drive, San Marcos, TX 78666}
\email{yang@txstate.edu}

\begin{document}

\maketitle

\begin{abstract}
In this paper, we study the product of orders of composition factors of odd order in a composition series of a finite linear group. First we generalize a result by Manz and Wolf about the order of solvable linear groups of odd order. Then we use this result to find bounds for the product of orders of composition factors of odd order in a composition series of a finite linear group.
\end{abstract}

\section{Introduction}
The order of a finite group is perhaps the most fundamental quantity
in group theory one can study. Accordingly, the concept of bounding
 the order of a finite group is a very natural one and has long been a
subject of vigorous research. For example, Manz and Wolf obtained
the following result \cite[Theorem 3.5]{manz/wolf} in bounding the
order of a solvable linear group by the size of the vector space on
which it acts. For the rest of this paper we let $\lambda=\sqrt[3]{24}$ and let $\alpha=(3
\cdot \log(48)+\log(24))/(3 \cdot \log(9)) \approx 2.25$.

\renewcommand{\labelenumi}{(\alph{enumi})}

\begin{thm}\label{ManzWolf}
    Let $G$ be a finite solvable group and let $V \neq 0$ be a finite, faithful, completely reducible G-module with $\cha(V)=p>0$. Then
    \begin{enumerate}
        \item $|G| \leq |V|^\alpha/ \lambda$.
        \item If $2 \nmid |G|$ or if $3 \nmid |G|$, then $|G| \leq |V|^2/\lambda$.
        \item If $2 \nmid |G|$ and $p \neq 2$, then $|G| \leq |V|^{3/2}/\lambda$.
    \end{enumerate}
\end{thm}

\renewcommand{\labelenumi}{(\arabic{enumi})}

In light of this result, it is natural to ask whether one can extend
(b) and (c) to a similar result for the order of a subgroup $H$ of a
completely reducible linear group $G$ (note that $H$ needs not to be completely reducible on $V$).

It should be pointed out that several recent advancements have improved the
previous theorem. For instance, Guralnick et al. \cite{GMP17} found
a bound for the product of abelian  composition factors of a
primitive permutation group, and Halasi and Mar\'{o}ti \cite{HM16}
generalized part (a) of the above theorem to
$p$-solvable groups.

Inspired by the above results and a sequence of papers written by
the fifth author \cite{KY18,QY19}, we consider the
product of the orders of certain abelian composition factors. By
combining the techniques used in \cite{KY18} and \cite{manz/wolf},
we obtain an upper bound for the product of the orders of the odd
order (abelian) composition factors of an arbitrary linear group,
which generalizes part of Theorem~\ref{ManzWolf} to an arbitrary
finite linear group.

We define $a(G)$ to be the product of orders of composition factors of odd order in a composition series of a finite group $G$. By the Jordan-H\"{o}lder's theorem, we see that this quantity is independent of the choice of the composition series.

Our main result is the following.

\begin{thm}\label{main1}
Let $G$ be a finite group acting on $V$ faithfully and completely reducibly where $V$ is of characteristic $p$. Then the following hold.
\begin{enumerate}
    \item $a(G)\leq |V|^2/\lambda$.
    \item If $p\neq 2$, then $a(G)\leq |V|^{3/2}/\lambda$.
\end{enumerate}
\end{thm}

The paper is organized as follows. In Section 2, we prove a slight generalization of \cite[Theorem 3.5(b)(c)]{manz/wolf} which includes the solvable case of Theorem ~\ref{main1}. In Section 3, we prove some properties of simple groups that are needed to reduce the general case to solvable groups. In Section 4, we prove a related result about permutation groups and then prove the main theorem of the paper.

We will use the following notation for the remainder of the paper.
All groups in this paper are assumed to be finite. Given a group
$G$, we use $F(G)$ to denote the Fitting subgroup of $G$, and use
$F^*(G)$ to denote the generalized Fitting subgroup of $G$. The
layer of $G$ is denoted as $E(G)$, and $\Out(G)$ is the outer
automorphism group of $G$. In addition, for a prime $p$, we denote
the order of Hall $p'$-subgroups of $G$ by $|G|_{p'}$.



\section{The Solvable Case}
In this section, we generalize parts (b) and (c) of \cite[Theorem 3.5]{manz/wolf} to a subgroup $H$ of $G$ that satisfies the respective conditions. We note that the action of $H$ on $V$ need not be completely reducible, and thus the generalization is not trivial.

\begin{prop}\label{newManzWolf}
     Let $G$ be a finite solvable group and let $V \neq 0$ be a finite, faithful, completely reducible G-module with $\cha(V)=p>0$. Let $H$ be a subgroup of $G$.
    \begin{enumerate}
        \item If $2 \nmid |H|$ or if $3 \nmid |H|$, then $|H| \leq |V|^2/\lambda$.
        \item If $2 \nmid |H|$ and $p \neq 2$, then $|H| \leq |V|^{3/2}/\lambda$.
    \end{enumerate}
\end{prop}


\begin{proof}  Since $G$ is solvable, we only need to consider the Hall $2'$-subgroup or the Hall $3'$-subgroup of $G$. The proof follows the arguments in \cite[Theorem 3.5]{manz/wolf} with some slight adjustments in each of the steps. For consistency we will adopt the notation used in \cite[Theorem 3.5]{manz/wolf}. Step 1 shows that $V$ is irreducible and the argument here is unchanged. Step 2 shows that $V$ is quasi-primitive and the argument is the unchanged as well. Step 3 shows that if we set $|V|=p^n$, then we may assume that $n \geq 2$ and $p^n \geq 16$. The calculation remains the same.

In Step 4, we show that $G \not\leq \Gamma(p^n)$, $n > 3$, and if $p=2$, then $n \geq 8$. All the arguments are the same with the exception of proving $n > 3$ for statement (2).
Assume $n=2$, we note that $e=2$, $2 \mid p-1$, and $p \geq 5$. We have
 \[|G|_{2'}\leq 1/2\cdot(p-1)\cdot 3\leq p^{3}/3\leq |V|^{3/2}/\lambda.\]
When $n=3$, we have $e=3$, $p \geq 7$, and thus $|V| \geq p^3$. Thus $|G| = |T||F/T||G/F| \mid (p-1)\cdot 9 \cdot 24$. We observe that $|G|_{2'}\leq \frac {p-1} 2 \cdot 27 \leq p^{4.5}/\lambda \leq  |V|^{3/2}/\lambda$.

In Step 5, by examining the proof of \cite[Theorem 3.5]{manz/wolf} carefully, we only need to check a few cases when $e$ is small for case (2).

\begin{enumerate}
    \item If $e=2$, then $|G|$ is divisible by $8$ and $A/F\leq \GL(2,2)$. Thus $|A/F|_{2'} \leq 3$. Since $|V|\geq 81$, we have
    \[ |G|_{2'} = (|G/A||A/F||F/T||T|)_{2'} \leq 3|U|^2 \leq 3|V| \leq \frac{|V|^{3/2}}{3}. \]

    \item If $e = 3$, then $|A/F|\leq \GL(2,3)$ and $p \geq 4$. Thus $|A/F|_{2'} \leq 3$. Since $|V|\geq 256$, we have
    \[ |G|_{2'}= (|G/A||A/F||F/T||T|)_{2'} \leq 27 \cdot |U|^2 \leq 27 \cdot  |V|^{2/3} \leq \frac{|V|^{3/2}}{3}. \]

    \item If $e=4$, we note that $|A/F|_{2'} \leq 15$ and $|V|\geq 81$. Thus we have
    $$|G|_{2'} \leq (|G/A||A/F||F/T||T|)_{2'} \leq 15 \cdot |U|^2\leq 15 \cdot |V|^{1/2} \leq \frac{|V|^{3/2}}{3}.$$
\end{enumerate}
This completes the proof.
\end{proof}


\section{Properties of Simple Groups}
The following property about the odd order subgroups of simple
groups is needed for the reduction of the main theorem to the
solvable case, which also has some applications to the study of
quantitative aspects of orbit structure of linear groups.

The general outline of the following proof is, for most finite simple groups of Lie type, we use results related to Zsigmondy primes to find two prime divisors $L_1$ and $L_2$ of $|G|$, such that there exist subgroups $H_1$, $H_2$ with $H_i = \Syl_{L_i}(G)$ satisfying the conditions required. There are some exceptional cases when either the rank or the size of the finite field is small. In these cases one cannot find suitable Zsigmondy primes. We handle these exceptional cases by checking the bounds via direct calculation.

\noindent

\begin{lem}\label{simpleodd}
Let $G$ be a finite non-abelian simple group and $r$ be a fixed prime. Then there exists a solvable subgroup $H$ of $G$ such that $(|H|,r) = 1$ and $|H|_{2'}\geq 2|\Out(G)|_{2'}$.
\end{lem}

\begin{proof}
 We now go through the Classification of Finite Simple Groups.
   \noindent
    \begin{enumerate}
        \item Let $G$ be one of the alternating groups $\Al_n$, $n\geq 5$. It's well-known that $|\Out(\Al_n)| = 2$ except when $n=6$ and $|\Out(\Al_6)| = 4$. Thus $|\Out(\Al_n)|_{2'} = 1$ . Since $5 \mid |\Al_n|$ and $3 \mid |\Al_n|$, the result follows.
        \vspace{0.1in}

        \item Let $G$ be one of the sporadic or the Tits groups. Then $|\Out(G)| \leq 2$ and the result can be confirmed by \cite{Conway}.

   For simple group of Lie type, we go through various families of Lie type. To illustrate the method, \cite[Proposition 4.1]{KY18} gives a detailed analysis for $\rm{A}$$_n(q)$ and shows how to handle most cases. For those finitely many exceptional cases, we will check that the required inequalities hold by direct calculation. Since these arguments are similar, for the remaining families of simple groups of Lie type, there is a table in \cite[Proposition 4.1]{KY18} that handles all the exceptional cases.
        \vspace{0.1in}
         \item Let $G=\rm{A}$$_1(q)$ where $q=p^f$. We have $$|G| = q(q+1)(q-1)d^{-1}$$ where $d= (2,q-1)$ and $|\Out(G)| = df$.

        Case (a). Suppose that $q$ is even. Then $d=1$ and $|\Out(G)| = f$.

        Assume there exists Zsigmondy prime $L_1$ for $p^{2f}-1$. Then  $L_1\mid p^{2f}-1$ and thus $L_1\mid p^f+1$ and $L_1\geq 2f = 2|\Out(G)| \geq 2|\Out(G)|_{2'}$.

        Assume there exists a Zsigmondy prime $L_2\mid p^f-1$, where $L_2\geq f$. It is clear that $L_1\neq L_2$. If $L_2 \geq 2f$, then we are done. Otherwise if $L_2^2 \mid p^f-1$, we consider the Sylow $L_2$-subgroup $L$. Then $|L|\geq 2f$. However, we have the following exceptions by  \cite[Lemma 3.1]{KY18}:

                \vspace{0.1in}
        (i) $f=4$, thus $|\Out(G)| = 4$, and $|\Out(G)|_{2'} = 1$. Since $2^4 + 1 = 17$ and $2^4-1 = 15 = 3\cdot 5$, we may choose $L_1 = 17$ and $L_2 = 5$.

        (ii) $f=6$, thus $|\Out(G)| = 6$, and $|\Out(G)|_{2'} = 3$. Since $2^6 + 1 = 65 = 5\cdot 17$ and $2^6-1 = 63 = 7\cdot 3^2$, we may choose $L_1 = 13$ and $L_2 = 7$.

        (iii) $f=12$, thus $|\Out(G)| = 12$, and $|\Out(G)|_{2'} = 3$. Since $2^{12}+1 = 4097 = 17 \cdot 241$ and $2^{12}-1 = 4095 = 3^2 \cdot 5 \cdot 7 \cdot 13$, we may choose $L_1 = 17$ and $L_2 = 7$.
                \vspace{0.1in}

        Case (b). Suppose that $q$ is odd. Then $d=2$ and $|\Out(G)| = 2f$, implying that $|\Out(G)|_{2'} = f$. We apply the same idea as before, $L_1 \mid p^f+1$ and $L_1\geq 2f = 2|\Out(G)|_{2'}$.

       There exists an $L_2\mid p^f-1$, where $L_2\geq f$, and $L_1\neq L_2$. If $L_2 \geq 2f$, then we are done. Otherwise if $L_2^2 \mid p^f-1$, we consider the Sylow $L_2$-subgroup $L$. Then $|L|\geq 2f$. The following case is the exception by \cite[Lemma 3.1]{KY18}:
               \vspace{0.1in}

       (i) When $p=3$, $f=4$, thus $|\Out(G)|_{2'} = 1$. Since $3^4+1 = 82 = 2\cdot 41$ and $3^4-1=80 = 2^4\cdot 5$, we may choose $L_1 = 41$ and $L_2 = 5$.
         \vspace{0.1in}

       \item Let $G=\rm{A}$$_n(q)$, where $q=p^f$ and $n\geq 2$. Set $m = \prod^n_{i=1}(q^{i+1}-1)$. Then $|G| = d^{-1}q^{n(n+1)/2}m$, $|\Out(G)| = 2fd$, where $d = (n+1,q-1)$.

       With the exception of a finite number of cases, there exists a Zsigmondy prime $L_1$ for $p^{f(n+1)} - 1$ such that $L_1\geq 2f(n+1)$, or $L_1^2 \mid p^{f(n+1)} - 1$. It follows that $L_1^2\geq 2f(n+1)$. Let $H_1$ be a Sylow $L_1$-subgroup $G$. Also, by \cite[Lemma 3.2]{KY18}, with the exception of a finite number of cases, there exists a Zsigmondy prime $L_2$ for $p^{fn} - 1$ such that $L_2\geq 3fn\geq 2f(n+1)$, or $L_2^2 \mid p^{fn} - 1$. This implies that $L_2^2\geq 3fn\geq 2f(n+1)$. Let $H_2$ be a Sylow $L_2$-subgroup of $G$. Notice that $L_1\neq L_2$.

       Since $|\Out(G)| = 2fd$, $|\Out(G)|_{2'} \leq fd$. Also, $n+1\geq d = (n+1, q-1)$, and $|H_1|, |H_2|\geq 2|\Out(G)|_{2'}$. Therefore, the result follows. The exceptions are listed in Table 1 (by \cite[Lemma 3.2]{KY18}).

\skipaline{}

    \begin{small}

    \begin{longtable}{|c|c|c|c|c|c|c|}    \hline
    $p$ & $n$ & $f$ & $d$ & $|\Out(G)|_{2'}$ & $|H_1|$ & $|H_2|$ \\
    \hline
    2  & 3,4,6,8,12,20  & 1  & 1  &   1       &  divides $2^{(n+1)}-1$    &  divides $2^n-1$    \\ \hline
    2 & 2  & 2  & 3  &  3   &   9   &  7    \\ \hline
    2 & 2  &  3 & 1  &   3   &   73   &  7    \\ \hline
    2 & 3  & 2  & 1  &    1 &   7   &  17    \\ \hline
    2 & 2  & 4  & 3  &   3   &  17    &  13    \\ \hline
    2 & 4  & 2  & 1  &   1  &   31   &   17   \\ \hline
    2 & 2  & 6  & 3  &  9   &   73   &   19   \\ \hline
    2 & 3  & 4  &1   &   1       &  257    & 17     \\ \hline
   2  & 4  & 3  & 1  &    3      &  151    &  31    \\ \hline
    2 &  6 & 2  & 1  &    1      &  127    &   43   \\ \hline
    2 & 2  & 10  & 3  &    15      &  331    &  151    \\ \hline
    2 &  4 & 5  & 1  &     5     &   31   & 11     \\ \hline
    2 & 5  & 4  & 1  &    1      &  31    &  11    \\ \hline
    2 & 10  & 2  & 1  &   1       &  31    &   11   \\ \hline
    3 & 2  & 2  & 1  &    1      &  7    &  5    \\ \hline
    3 & 2  & 3  & 1  &  3        & 13     & 7     \\ \hline
    3 & 3  & 2 & 4  &  1        & 41     & 13     \\ \hline
    \end{longtable}
     \setcounter{table}{0}
    \captionof{table}{Exceptional cases for $\rm{A}$$_n(q)$} \label{tab:title}
        \end{small}

\item Let $G=$ $\rm{^2A}$$_n(q^2)$ where $n\geq 2$. Note that if $n=2$, then $q>2$. Set $m=\prod^n_{i=1}(q^{i+1} - (-1)^{i+1})$, $q^2 = p^f$ and $d=(n+1,q+1)$. Then $|G| = d^{-1}mq^{n(n+1)/2}$, and $|\Out(G)| = df$. By  \cite[Theorem A]{WF88} and \cite[Lemma 3.2]{KY18}, there exists a Zsigmondy prime $L_1 \mid p^{f(n+1)/2}-(-1)^{n+1}$ such that $L_1\geq 2(n+1)f\geq 2df$ or $L_1^2 \mid p^{f(n+1)/2}-(-1)^{n+1}$ and $L_1^2\geq 2(n+1)f\geq 2df$. Moreover, by \cite[Theorem A]{WF88} and \cite[Lemma 3.2]{KY18}, with the exception of a finite number of cases, there exists a Zsigmondy prime $L_2\mid p^{fn/2}-(-1)^{n+1}$ such that $L_2\geq \frac 5 2 (n+1)f\geq 2(n+1)f  \geq 2df$ or $L_2^2\mid p^{fn/2}-(-1)^{n+1}$ and $L_2^2\geq \frac 5 2 (n+1)f\geq 2(n+1)f  \geq 2df$. For all the exceptional cases, we can check the results via direct calculation and by considering the order of the group $G$ and $\Out(G)$ (see Table 2).

    \begin{small}
\begin{longtable}{|c|c|c|c|c|c|c|} \hline
    $p$ & $n$ & $f/2$ & $d$ & $|\Out(G)|_{2'}$ & $|H_1|$ & $|H_2|$ \\ \hline
    2& 3  & 1 &  1 &    1      &  9    &    5  \\ \hline
    2& 2  & 2 &  1 &     1     &13     &    5  \\ \hline
    2& 4  & 1 &   1&    1      &  13    &   5   \\ \hline
    2& 2  &3  &  3 &   9       &  243    &   19   \\ \hline
    2& 3  &2  &1   &   1       &    13  &     5 \\ \hline
    2& 6  & 1 & 1  &  1        &   43   &   7  \\ \hline
    2& 2  & 4 &  1 &     1     &   241   &   17   \\ \hline
    2& 4  &2  & 5  &     5     &    41  &  25    \\ \hline
    2& 8  &1  &3   &  3        &    19  &   17   \\ \hline
    2&  2 & 5 & 3  &  15        &    331  &   31   \\ \hline
    2&  5 & 2 &  1 &   1       &   7   &     5 \\ \hline
    2&  10 & 1 & 1  &    1      &    31  &  11    \\ \hline
    2& 2  &6  & 1 &     3     &   37   &    13  \\ \hline
    2& 3  &4  &1   &     1     &  241    &    17  \\ \hline
    2&4   &  3&  1 &        3  &  13    &      11\\ \hline
    2&  6 & 2 &  1 &    1      &    7  &    5  \\ \hline
    2&  12 &1  &1   &  1        &   7   &    5  \\ \hline
    2&   2& 9 & 3  &     27     &   87211   &  73    \\ \hline
    2 &  3 & 6 & 1  &    3      &  37    &   13   \\ \hline
   2 &  6 & 3 & 1  &   3       &   19   &    7  \\ \hline
    2& 9  & 2 &  5 &   5       &    41  &    31  \\ \hline
    2& 18  &1 &  1 &    1      &   19   &   7   \\ \hline
    2&  2 & 10 &  1 &   5       &   61   &  41    \\ \hline
    2&  4 & 5 &  1 &    5      &   31   &   11   \\ \hline
    2&  5 & 4 & 1  &     1     &  13    &  9    \\ \hline
     2&  10 & 2 & 1  &   1       &  31    & 11     \\ \hline
    2&  20 &1  & 3  &     3     &   41   &   31   \\ \hline
    2& 2  & 14 & 1  &    7      &   127   &   43   \\ \hline
    2& 4  & 7  & 1  &    7      &   71   &   43   \\ \hline
    2&7   & 4 & 1  &    1      &   257   &  17    \\ \hline
    2&14   & 2 & 5  &    5      &  41    &   17   \\ \hline
    2& 28  & 1  & 1  &    1      &  59   &  17    \\ \hline
    3& 3  & 1 & 4  &    1      &  5   &   13   \\ \hline
    3& 2  & 2 &  1 &     1     & 73     &    5  \\ \hline
    3& 4  & 1 &   1&    1      &  61    &   5   \\ \hline
    3& 2  &3  &  1 &   3       &  13    &   7   \\ \hline
    3& 3  &2  &2   &   1       &    73  &     41 \\ \hline
    3& 6  & 1 & 1  &  1        &   13   &   7  \\ \hline
    5& 2  & 2&  1 &      1    &    601  &   13   \\ \hline
    5& 4  & 1 & 1  &    1      &  3   &   313   \\ \hline
 \end{longtable}
 \setcounter{table}{1}
 \captionof{table}{Exceptional cases for $\rm{^2A}$$_n(q^2)$} \label{tab:title}
   \end{small}
\end{enumerate}

    We now provide a table (Table 3) for the simple groups of Lie types other than $\rm{A}$$_n(q)$ and $\rm{^2A}$$(q^2)$. The second and the third columns in the table are two large prime divisors that correspond to Sylow subgroups.

\begin{tiny}
\begin{flushleft}
\begin{longtable}{ |c|c|c|c|} \hline
type & $L{_1}$ & $L{_2}$ &exceptional cases \\\hline
   $B_n(q),q=p^f$  &  $p^{f(2n)}-1$  & $p^{f(2n-2)}-1$  &                  $(f=1,n=2); (p=2, f=3, n=2); (p=2, f=1, n=3)$        \\\hline
   $C_n(q),q=p^f$ &  $p^{f(2n)}-1$  & $p^{f(2n-2)}-1$ & $(f=1,n=2)$;$(p=2,f=3,n=2)$; $(p=2,f=1,n=3)$                         \\\hline
   $D_n(q),q=p^f$  &  $p^{f(2n-2)}-1$ & $p^{f(2n-4)}-1$  & $(p=2,f=1,n=5)$;$(p=2,f=2,n=5)$;$(p=2,f=1,n=3)$;                      \\
    & & &  $(p=2,f=1,n=4)$; $(p=2,f=3,n=4)$;\\
     & & &$(p=3,f=1,n=4)$;$(p=5,f=1,n=4)$ \\\hline
   $^2D_n(q^2),q^2=p^f$  & $p^{f(n-1)}-1$  & $p^{f(n-2)}-1$ &    $(p=3,f=2,n=3, 4)$;$(p=5,f=2,n=4)$;                 \\\hline
   $E_6(q),q=p^f$  &   $p^{12f}-1$ & $p^{8f}-1$  &                         \\ \hline
    $E_7(q),q=p^f$ &  $p^{18f}-1$  & $p^{14f}-1$ &                         \\\hline
    $E_8(q),q=p^f$  &   $p^{30f}-1$ & $p^{24f}-1$  &                         \\ \hline
    $F_4(q),q=p^f$ &   $p^{12f}-1$ & $p^{8f}-1$  &                         \\ \hline
    $G_2(q),q=p^f$ &   $p^{6f}-1$  & $p^{2f}-1$&  $(p=2, f=1,2,3)$ \\\hline
    $^2E_6(q),q^2=p^f$ & $p^{6f}-1$   & $p^{4f}-1$ &                         \\ \hline
    $^3D_4(q^3),q^3=p^f$ & $p^{4f}-1$ & $p^{2f}-1$  &                     $(p=2, f=3,5,7)$    \\\hline
   $^2B_2(2^{2n+1})$ & $2^{4(2n+1)}-1$  & $2^{(2n+1)}-1$  &                         \\ \hline
    $^2F_4(2^{2n+1})$ & $2^{4(2n+1)}-1$ &  $2^{(2n+1)}-1$  &                         \\\hline
    $^2G_2(3^{2n+1})$ & $3^{3(2n+1)}+1$  & $3^{(2n+1)}-1$  &                         \\ \hline
\end{longtable}
 \setcounter{table}{2}
 \captionof{table}{Other Lie Type Groups} \label{tab:title}
\end{flushleft}
\end{tiny}
\end{proof}

\section{Composition factors of odd order}

In this section we prove the main result of the paper. Before doing so, we need the following proposition about permutation groups.

\begin{prop}\label{oddperm}
Let $G$ be a group of permutations on a set $\Omega$ of order $n$. Then $a(G)$ is at most $2^{n-1}$.
\end{prop}
 \begin{proof}
 We first check that the result is true for $n \leq 4$ $(|\Sy_2|_{2'}\leq 2, |\Sy_3|_{2'}\leq 4,|\Sy_4|_{2'}\leq 8)$. We may assume that $n \geq 5$. We shall proceed by induction on $|G|$.

We first suppose that $G$ is intransitive on $\Omega$. Write $\Omega= \Gamma_1\cup\Gamma_2$ for nonempty subsets $\Gamma_i$ of
$\Omega$ such that $G$ permutes $\Gamma_1$ and also $\Gamma_2$. Write $n_i=|\Gamma_i|$ for $i=1,2$; so clearly $n_1+n_2=n$. Let $L$ be the kernel of $G$ on $\Gamma_1$, so that $L$ acts faithfully on $\Gamma_2$. By induction we then have $a(G/L) \leq {2}^{n_1-1}$ and $a(L) \leq {2}^{n_2-1}$, and so $a(G) \leq  {2}^{n_1-1}  \cdot {2}^{n_2-1} = {2}^{n-2} < {2}^{n-1}$, as desired. So now we may assume that $G$ is transitive on $\Omega$.

Suppose that $G$ is imprimitive on $\Omega$. Then there is a nontrivial decomposition $\Omega = \bigcup_i \Omega_i$ with $G$ permuting
$X = \{\Omega_1, \dots , \Omega_r \}$ and $N_G(\Omega_1)$ acting primitively on $\Omega_1$ where $|\Omega_1|=m$ and $n=m r$.
Let $\pi$ be the permutation representation of $G$ on $X$ and $K = \ker \pi$. Set $K_0=K$ and $K_{i+1} =\{g \in K_i \ |\ g$ acts trivially on $\Omega_{i+1} \}$.
We note that $a(G) = a_0 a_1 \dots a_r$, where $a_0 =a(G/K)$ and $a_i = a(K_{i-1}/K_i)$ for $i \geq 1$. By induction, $a_0 \leq 2^{r-1}$ and $a_i \leq 2^{m-1}$ for $i \geq 1$. It is easy to see that $a(G) \leq 2^{n-1}$.

Thus we may assume $G$ is primitive, then we know that $G$ either contains $\Al_n$ or is one of the groups in the exceptional list by \cite[Corollary 1.4]{M02}. If $G$ contains $\Al_n$, then the result is clear since $n\geq 5$. Otherwise, $G$ is one of the groups in the exceptional list, we verify the result in Table 4.

    \begin{small}
     \begin{center}
    \begin{longtable}{|c|c|c|c|}\hline
         $G$ & $|G|_{2'}$ & $n$ & $2^{n-1}$ \\\hline
         ${\rm{AGL}(1,5)}$& 5  & 5 & 16   \\\hline
          $\rm{AGL}(3,2)$& 21 & 8  &128    \\\hline
          $\rm{AGL}(2,3)$& 27  &9  &256   \\\hline
         $\rm{AGL}(4,2)$&315  &16  & 32768   \\\hline
          $\rm{A\Gamma L}(1,8)$& 21  &8  &128   \\\hline
         $2^4:\Al_7$ & 315 &16  &  32768  \\\hline
         $\rm{PSL}(2,5)$ & 15& 6&32 \\\hline
         $\rm{PSL}(3,2)$ & 21 &7&64\\\hline
         $\rm{PSL}(2,7)$&21&8&128\\\hline
         $\rm{PSL}(3,3)$&351 &13 &4096 \\\hline
         $\rm{PSL}(4,2)$&315 &15 &16384\\\hline
         $\rm{PGL}(2,5)$&15 &6 &32\\\hline
         $\rm{PGL}(2,7)$&21 & 8 & 128\\\hline
         $\rm{PGL}(2,9)$&45 &10 & 512\\\hline
         $\rm{P\Gamma L}(2,8)$&189 & 9&256 \\\hline
         $\rm{P\Gamma L}(2,9)$& 45& 10 &512 \\\hline
         $M_{10}$&45 & 10 & 512 \\\hline
         $M_{11}$& 495& 11 &1024  \\\hline
         $M_{11}$&495 &12&2048 \\\hline
         $M_{12}$&1485 &12 &2048  \\\hline
         $M_{23}$&79695 &23 &4194304  \\\hline
         $M_{24}$&239085 &24 &8388608  \\\hline
         $S_6$ &45 &  10&512 \\\hline

   \end{longtable}
 \setcounter{table}{3}
     \captionof{table}{List of exceptional primitive groups not containing $\Al_n$ } \label{tab:title}
         \end{center}
     \end{small}
 \end{proof}
With this result and the work done in the previous sections, we now can prove the main result.


\noindent {\emph{Proof of Theorem ~\ref{main1}.}}  Let $S$ be the
maximal normal solvable subgroup of $G$. Consider $\bar{G} = G/S$.
It is easy to see that $F(\bar{G}) = 1$. Therefore, $F^*(\bar{G}) =
F(\bar{G})E(\bar{G}) = E(\bar{G})$. Also, we know that
$Z(E(\bar{G}))$ is trivial, otherwise $S$ is not maximal. Since
$E(\bar{G})/Z(E(\bar{G}))$ is the unique largest semi-simple
subgroup of $\bar{G}$, $E(\bar{G})$ is the product of  simple
non-abelian subgroups.

 Let $\bar{C} = C_{\bar{G}}(E(\bar{G}))$. Since $F^*(G)$ is self-centralizing, $\bar{C}<F^*(G)$.
 Let $K = \bar{G}/\bar{C}$. Then $K$ acts faithfully on $E(\bar{G})$.
 We may assume that $K$ acts transitively on $$L_1 = E_{11}\times\dots \times E_{1k_1},$$
where $E_{11}, \ldots, E_{1k_1}$ are non-abelian simple components
of $E(\bar{G})$. Let $K_1 = C_{K}(L_1)$. We may assume that $K_1$
acts transitively on $$L_2 = E_{21}\times\dots\times E_{2k_2},$$
where $E_{21}, \ldots, E_{2k_2}$ are non-abelian simple components
of $E(\bar{G})$. Let $K_2 = C_{K_1}(L_2)$,
 inductively, we may define $L_3, K_3, \ldots, L_t, K_t$. Then $E(\bar{G}) = L_1\times\dots\times L_t$, and $K_{i-1}/K_{i}$ acts transitively on $L_i$.

Since $G$ acts on $V$ completely reducibly, and $S$ is a normal
subgroup of $G$, we know that $S$ acts on $V$ completely reducibly.
Since $E_1, \ldots, E_m$ are non-abelian simple subgroups, by Lemma
~\ref{simpleodd} there exists a solvable subgroup $H = H_1\times
\dots \times H_m$, where $H_i< E_i$ such that $(H_i, p) = 1$, and
$|H_i|_{2'} \geq 2|\Out(E_i)|_{2'}$. Therefore, $|H|_{2'} =
\prod_i|H_i|_{2'}\geq 2^m|\Out(E(\bar{G}))|_{2'}$.

Moreover, $K$ is a permutation group permuting
$E_1,\ldots,E_m$. By Proposition ~\ref{oddperm}, the
product of the orders of all the odd order composition factors of
$K$ is less than $2^{m-1}$. Thus, $|H|_{2'}\geq
2^{m}|\Out(E(\bar{G}))|_{2'}$,  which is greater than the product of
the orders of all the odd order composition factors of $K$ and the $2'$ part of the outer automorphism of $E(\bar{G})$.

Let $\phi: G\rightarrow G/S$ to be the canonical homomorphism. Since $(|H|,p) = 1$, we know that $\phi^{-1}(H)$ acts on $V$ completely reducibly by the generalized Maschke's Theorem (cf. \cite[Problem 1.8]{IS94}). Therefore, by Proposition ~\ref{newManzWolf}, we have $|\phi^{-1}(H)|_{2'}\leq |V|^2/\lambda$, and if $p\neq 2, |\phi^{-1}(H)|_{2'}\leq |V|^{3/2}/\lambda$.

We observe that the odd order composition factors of $G$ are
distributed in the maximal normal solvable subgroup $S$, the outer
automorphism of the direct product of simple groups $E(\bar{G})$,
and the odd order composition factors of the permutation group
$\bar{G}/\bar{C}$. Therefore, $a(G) \leq
|\phi^{-1}(H)|_{2'}$, and the result follows.\qed

From this result we derive the following corollary.

\begin{cor}\label{main2}
Let $G$ be a finite group acting on $V$ faithfully and completely reducibly ($V$ is possibly of mixed characteristic). Then $a(G)\leq |V|^2/\lambda$.
\end{cor}


\section*{Acknowledgement}

This research was conducted under NSF-REU grant DMS-1757233 by the
first, second, third, and forth authors during
the Summer of 2019 under the supervision of the fifth author.
The authors gratefully acknowledge the financial support of NSF and also
thank Texas State University for providing a great working environment
and support.

\end{document}